\documentclass[10pt,a4paper]{article} 
\usepackage{amsmath,amssymb,amsthm,amsfonts,amscd,euscript,verbatim, t1enc, newlfont, graphicx, cancel}
\usepackage{hyperref}
\usepackage[all]{xy}
\providecommand{\keywords}[1]{\textbf{\textit{Key words and phrases }} #1}
\providecommand{\subjclass}[1]{\textbf{\textit{2010 Mathematics Subject Classification.}} #1}

\theoremstyle{definition}
\newtheorem{theo}{Theorem}[section]

\newtheorem{pr}[theo]{Proposition}

 \newtheorem{lem}[theo]{Lemma}
 \newtheorem{coro}[theo]{Corollary}
  
\theoremstyle{remark}
\newtheorem{rema}[theo]{Remark}

\theoremstyle{definition}
\newtheorem{defi}[theo]{Definition}

\numberwithin{equation}{section}

\newcommand\cu{\underline{C}}
\newcommand\du{\underline{D}}

\newcommand\bu{\underline{B}}
\newcommand\hu{\underline{H}}

\newcommand\obj{\operatorname{Obj}}

\newcommand\id{\operatorname{id}}
\DeclareMathOperator\adfu{\operatorname{AddFun}}
\DeclareMathOperator\adfur{\operatorname{Fun}_R}

\DeclareMathOperator\kar{\operatorname{Kar}}
 \DeclareMathOperator\ke{\operatorname{Ker}}

\newcommand\hw{{\underline{Hw}}}
\newcommand\hrt{{\underline{Ht}}}

\newcommand\alo{{\aleph_1}}

\newcommand\modd{\operatorname{Mod}}

\newcommand\z{{\mathbb{Z}}}

 \newcommand\lan{\langle}
\newcommand\ra{\rangle}

\newcommand\al{\alpha}

\newcommand\ns{\{0\}}
\newcommand\ab{\operatorname{Ab}}

\newcommand\cp{\mathcal{P}}
\newcommand\cq{\mathcal{Q}}

\newcommand\opp{^{op}}
\newcommand\brab{]^{\operatorname{cl}}}

\begin{document}

\title
 {On weakly negative subcategories,  weight structures, and (weakly) approximable triangulated categories}
\author{Mikhail V. Bondarko, Sergei V. Vostokov
   \thanks{ 
	This work was funded by the Russian Science Foundation under  grant no. 16-11-
00200.}}\maketitle
\begin{abstract} 
In this note we 
 prove that certain triangulated categories are (weakly) approximable in the sense 
of A. Neeman. We prove that a triangulated $\cu$ that is compactly generated by a single object $G$ is weakly approximable if $\cu(G,G[i])=\ns$ for $i>1$ (we say that $G$ is weakly negative 
if this assumption is fulfilled; the case where the equality $\cu(G,G[1])=\ns$ is fulfilled as well was mentioned by Neeman himself). Moreover, if $G\cong \bigoplus_{0\le i\le n}G_i$ and $\cu(G_i,G_j[1])=\ns$ whenever 
 $i\le j$ then $\cu$ is also approximable. 

 The latter result can be useful since (under a few more additional assumptions) it allows to characterize a certain explicit subcategory of $\cu$ 
as the category of {\it finite} cohomological functors from the subcategory $\cu^c$ of compact objects of $\cu$ into 
 $R$-modules (for a noetherian commutative ring $R$ such that $\cu$ is $R$-linear). One may apply 
 this statement to the construction of certain adjoint functors and $t$-structures. 

Our proof of (weak) approximability of $\cu$ under the aforementioned assumptions is 
closely related to  ({\it weight decompositions} for) certain {\it (weak) weight structures}, and we discuss this relationship in detail. 
   
\end{abstract}



 \section*{Introduction}
The goal of this note is to prove that certain triangulated categories are {\it (weakly) approximable} 
in a sense introduced by A. Neeman. The importance of this property comes from Theorem 0.3 of \cite{neesat}. We will say more on  that theorem 
 in \S\ref{sappl} below; now we only recall that 
 is says the following: for an approximable $R$-linear triangulated category $\cu$ (where $R$ is a noetherian coefficient ring)  that satisfies certain conditions one can describe explicitly a subcategory $\cu_c^b\subset \cu$ such that for the triangulated subcategory $\cu^c$ of compact objects of $\cu$ the restricted Yoneda functor $\mathcal{Y}:\cu\to \adfur(\cu^c{}\opp,R-\modd)$ (the latter is the category of $R$-linear functors into $R$-modules, whereas $\mathcal{Y}$ sends an object $N$ of $\cu$ into the restriction of the functor it represents to $\cu^c$)  gives a full embedding of $\cu_c^b$ into $\adfur(\cu^c{}\opp,R-\modd)$. 
 Moreover,    
 a  cohomological functor $H: \cu^c{}\opp\to R-\modd$ belongs to the image of $\mathcal{Y}$ whenever the $R$-module $\bigoplus_{i\in \z}H(M[i])$ is finitely generated for any $M\in \obj \cu^c$. This ("Brown-type") representability result has nice applications to the construction of certain adjoint functors (see Corollary 0.4 and Remark 0.7 of ibid.) and $t$-structures (in particular, for derived categories of quasi-coherent sheaves). Moreover, 
 some statements related to loc. cit. are also fulfilled if $\cu$ is only weakly approximable (in the sense of ibid.).

Now, Remark 3.3 of ibid. states that $\cu$ is approximable whenever it is {\it compactly generated} by an object $G$ such that $G\perp G[i]$ for all $i>0$ (i.e., there are only zero $\cu$-morphisms between $G$ and $G[i]$ in this case; we will say that the set $\{G\}$ is {\it negative} in $\cu$). 
The authors' attention to this matter was attracted by the following observation: for any $(\cu,G)$ satisfying these conditions there exists a {\it weight structure} $w=(\cu_{w\le 0},\cu_{w\ge 0})$ on $\cu$ such that $\cu_{w\le 0}$ (resp. $\cu_{w\ge 0}$) is the smallest class of objects of $\cu$ that is closed with respect to extensions and coproducts, and contains $G[i]$ for all $i\le 0$ (resp. for $i\ge 0$; see \S\ref{swws} below). Moreover,  in this case the "natural" choices of distinguished triangles in the definition of approximability (see Definition 0.21(ii) of ibid.) are just {\it weight decompositions} of elements of $\cu_{w\ge 0}$. This weight decomposition argument appears to be the optimal way to justify the aforementioned Remark 3.3 of ibid., and it is closely related to the proof of   \cite[Theorem 2.2.1(1)]{bsnew}. Furthermore, if one weakens the negativity assumption on $G$ and demands that  $G\perp G[i]$ for $i>1$ only (we call this condition the {\it weak negativity} one) then one can still get a {\it weak weight structure} on $\cu$ (see Remark \ref{rstws}(\ref{iweak}) below), and obtain that $\cu$ is  weakly approximable. Under 
 the additional assumption 
 that $G\cong \bigoplus_{0\le i\le n}G_i$, where $G_i\perp G_j[1]$ for $i\le j$ (this is certainly weaker than the negativity of $\{G\}$)  we are 
 able to prove that $G$ is approximable; see Corollary \ref{capprox}(2) and Remark \ref{rex2}(3). 

The authors believe that the weight structure view on this matter can be useful since weight structures "glue" and "localize" well (one may say that "weight structures localize better than $t$-structures"); see Remark \ref{rstws}(\ref{iwloc}). Note also that in Theorem \ref{t1} we also consider the case where $\cu$ is (compactly) generated by a not necessarily 
 small weakly negative class of compact objects.

Let us  now describe the contents  of the paper. 

In \S\ref{snotata}  we introduce some categorical definitions and notation. 

In \S\ref{smash} we recall some more definitions (mostly related to triangulated categories that are {\it smashing}, i.e., closed with respect to coproducts) and also a few lemmas.

In \S\ref{smain} we formulate our main results in terms of certain extension-closures, and discuss examples. 

In \S\ref{sappr} we translate the central Theorem \ref{t1} into the language used in \cite{neesat} (this includes $t$-structures) and obtain a statement on (weak) approximability. We also demonstrate   that $\cu$ is (weakly approximable but) not necessarily approximable if the set $\{G\}$ is only weakly negative.

In \S\ref{sappl} we recall 
 the relation of approximability to 
the aformentioned Brown representability-type Theorem 0.3 of ibid. 

In \S\ref{swws} we recall the notions of a weight structure and a weak weight structure, and discuss their relationship to our results (along with their properties).

The authors are deeply grateful to prof. A. Neeman for his very interesting talk made in St. Petersburg, that inspired the current note.

\section{Some categorical notation and definitions}\label{snotata}

\begin{itemize}

\item All 
 coproducts in this paper will be small.

\item Given a category $C$ and  $X,Y\in\obj C$  we will write $C(X,Y)$ for  the set of morphisms from $X$ to $Y$ in $C$.

\item For categories $C'$ and $C$ we write $C'\subset C$ if $C'$ is a full 
subcategory of $C$.

\item Given a category $C$ and  $X,Y\in\obj C$, we say that $X$ is a {\it
retract} of $Y$ 
 if $\id_X$ can be 
 factored through $Y$.\footnote{Certainly,  if $C$ is triangulated 
then $X$ is a retract of $Y$ if and only if $X$ is its direct summand.}\ 

\item A 
 class of object $D$ in an additive category $C$ 
is said to be {\it retraction-closed} in $C$ if it contains all $C$-retracts of its elements. 

\item  For any $(C,D)$ as above we will write 
$\kar_{C}(D)$ 
for the class of all $C$-retracts of 
 elements of $D$. 

\item The symbol $\cu$ below will always denote some triangulated category. 
 Moreover, we will always assume that $\cu$ is a {\it smashing},  
 that is, it is closed with respect to (small) 
 coproducts.


\item For $X,Y\in \obj \cu$ we will write $X\perp Y$ if $\cu(X,Y)=\ns$. 

For
$D,E\subset \obj \cu$ we write $D\perp E$ if $X\perp Y$ for all $X\in D,\
Y\in E$.


\item For any  $A,B,C \in \obj\cu$ we will say that $C$ is an {\it extension} of $B$ by $A$ if there exists a distinguished triangle $A \to C \to B \to A[1]$.

\item A class $\cp\subset \obj \cu$ is said to be  {\it extension-closed}
    if it 
		is closed with respect to extensions and contains $0$.  

\end{itemize}


\section{Some definitions and lemmas for (smashing) triangulated categories}\label{smash}

Let us give some more definitions.

\begin{defi}\label{dcomp}
Let $\cu$ be a smashing triangulated category, $M\in \obj \cu$, and $\cp,\cp'\subset \obj \cu$.
 
\begin{enumerate}

\item\label{idstar}
We will write $\cp\star \cp'$  for the class of all extensions of elements of $\cp'$ by elements of $\cp$.

We will say that $\cp$ is {\it smashing} (in $\cu$) if it is closed with respect to $\cu$-coproducts.

\item\label{icomp}
We will say that $M$ is {\it compact} (in $\cu$) if  the functor $\cu(M,-):\cu\to \ab$ respects coproducts.

\item\label{ibr}
We will write  $ [\cp\brab$ for the smallest smashing extension-closed subclass of $\obj \cu$ that contains $\cp$, and call this class the {\it big extension-closure} of $\cp$. 

\item\label{icompg}
We will say that $\cp$ {\it compactly generates} $\cu$ 
if 
elements of $\cp$ are compact and  the class  $ [\cup_{i\in\z}\cp[i]\brab$ equals the whole $\obj \cu$.\footnote{We do not have to assume that this class $\cp$ is a set for the purposes of this paper. Note however that the main application of this theorem is Corollary \ref{capprox}, where we assume $\cp$ to be finite.}

\end{enumerate}
\end{defi}

\begin{pr}\label{pstar}
Assume that $A$ and $B$ 
 are some classes of objects of $\cu$. 

\begin{enumerate}
\item\label{iort}
Assume that $A\perp B$. Then $[A \brab\perp B$ as well.

\item\label{iortc}
Assume that $A\perp B$ and all elements of $A$ are compact. Then $[A \brab\perp [B \brab$ as well.

\item\label{iass}
 $(X \star Y) \star Z = (X\star Y) \star Z$ for any classes $X,Y,Z\subset \obj \cu$, that is, the operation $\star$ is associative on subclasses of $\obj \cu$. 

\item\label{istar}
 Assume that $A$ and $B$ are extension-closed and  $A\perp B[1]$. Then the class $A\star B$  
 is extension-closed as well.

\item\label{istarsm} 
Assume 
that 
  $A$ and $B$ are smashing. Then  $A\star B$  is smashing as well.

\item\label{istarkar} 
$\kar_{\cu}(A)\star \kar_{\cu}(B)\subset \kar_{\cu}(A\star B)$.

\end{enumerate}
\end{pr}
\begin{proof}
All of these statements are rather easy.

Assertions \ref{iort} and \ref{iortc} are obvious. 

Assertion \ref{iass} is essentially given by Lemma 1.3.10 of \cite{bbd}.

Assertions \ref{istar} and \ref{istarsm} immediately follow from  Proposition 2.1.1 of \cite{bsnew} (cf. also Remark 1.2.2 of \cite{neebook}). 

\ref{istarkar}. If $M\to X\to N\to M[1]$ is a distinguished triangle, where $M\bigoplus M'=M''\in A$ and $N\bigoplus N'=N''\in B$ then we can add it with the split distinguished triangle $M'\to M'  \bigoplus N'\to N'\to M'[1]$ to get a triangle  $M''\to M'  \bigoplus N'\bigoplus X\to N''\to M''[1]$. This obviously yields the result.
\end{proof}

\section{Our main results in terms of big extension-closures}\label{smain}

The following simple definition is important for this paper.

\begin{defi}\label{dneg}
We will say that a class of objects $\cp\subset\obj \cu$ is {\it weakly negative} (in $\cu$) if $\cp \perp (\cup_{i>1}\cp[i])$.

Moreover, we will  say  that $\cp$ is {\it negative} if we also have $\cp\perp \cp[1]$.
\end{defi}

Now we are able to establish our central results rather easily. We start from formulating them in terms of big extension-closures; we will relate them to $t$-structures (as mentioned in Definition 0.21 of \cite{neesat}) and to (weak) weight structures later. 

\begin{theo}\label{t1}
Let $\cp$ be a weakly negative class of compact objects of (a smashing triangulated category) $\cu$; 
 set $A=[\cp \brab$  and $B= [\cup_{i\ge 0}\cp[i] \brab$. 

1. Then $B=A\star B[1]$.

2. Assume that $\cp=\cup_{0\le i\le n}\cp_i$ (for some $n\ge 0$) and that for these subclasses 
we also have $\cp_i\perp \cp_j[1]$ whenever $i\le j$. Then for the classes $\cq_i$  consisting of all coproducts of elements of $\cp_i$ we have $\cq_0\star \cq_1\star\dots \star\cq_n=A$ (see Proposition \ref{pstar}(\ref{iass})).

\end{theo}
\begin{proof}
1. According to Proposition \ref{pstar}(\ref{iortc}), $A\perp B[2]$. Applying parts \ref{istar} and \ref{istarsm} of this proposition we obtain that the class $A\star B[1]$ is smashing and extension-closed. Since it contains $\cp[i]$ for all $i\ge 0$, we obtain $B\subset A\star B[1]$. Lastly, this inclusion is an equality since $B$ is extension-closed. 

2. We prove the statement by induction on $n$.

 In the case $n=0$ it suffices to verify that the class $\cq$ of all coproducts of elements of $\cp_0$ is extension-closed.  This fact certainly reduces to $\cq\perp \cq[1]$, and the latter orthogonality is   immediate from Proposition \ref{pstar}(\ref{iortc}).

Now, let us assume that the assertion is fulfilled for all $n<n_0$, where $n_0>0$. 
For $n=n_0$ we argue similarly to the proof of assertion 1. The inductive assumption implies that the class 
  $\cq'=\cq_0\star \cq_1\star\dots \star\cq_{n-1}$ is smashing and extension-closed, and the same is true for $\cq_n$.
	Since $\cq'\perp \cq_n[1]$ immediately from Proposition \ref{pstar}(\ref{iortc}), applying parts \ref{istar} and \ref{istarsm} of that proposition we obtain that the class $A'=\cq_0\star \cq_1\star\dots \star\cq_n$ is smashing and extension-closed as well (here we apply once again the associativity provided by Proposition \ref{pstar}(\ref{iass})). It obviously follows that $A'=A$.
		\end{proof}

\begin{rema}\label{rex1}
1. To obtain an example for our theorem  one can start from a small\footnote{The authors suspect that smallness is not really necessary here; cf. Remark 2.2.2(2) of \cite{bsnew}.} {\it differential graded} category 
$\bu$ (see \S2.2 of \cite{dgk} or \S2.1 of \cite{mymot}) with $\obj \bu=\cp$ and such that the complex $\bu^*(M,N)$ is acyclic in 
degrees bigger than $1$; certainly, there exist plenty of categories of this sort (cf. part 2 of this remark). 
 Moreover, to obtain an example for Theorem \ref{t1}(2) we assume in addition that the first cohomology of $\bu^*(M,N)$  also vanishes if $M\in \cp_i$, $N\in \cp_j$, and $i\le j$ (recall that we assume $\cp$ to be equal to $\cup_{0\le i\le n}\cp_i$). Then the well-known statements mentioned in  \S3.5 of ibid. imply that we can take $\cu$ to be the {\it derived category} of $\bu$. 

More generally, one can also consider {\it spectral} examples; see Theorems  3.1.1 and 3.3.3 of \cite{schwmod}. 
We are mostly interested in the case where $\cp$ is finite (see \S\ref{sappr}) 
 and $\cu$ is   {\it monogenic}; being more precise, we assume that $\cu$ is compactly generated by the set  $\{G\}$, where $G=\bigoplus_{P\in \cp}P$. The  endomorphism spectrum of 
$G$ should have zero homotopy groups in degrees less than $-1$ (i.e., it is {\it $-2$-connected}); this corresponds to weak negativity of $\{G\}$. 
 The case where 
  $\{G\}$ is negative  (that corresponds to $n=0$ in part 2 of our theorem) is closely related to the so-called 
 connective stable homotopy theory as discussed in \S7 of \cite{axstab}; cf. also Remark 4.3.4(2) of \cite{bwcp}.

2. We will say more about examples and on their relationship to weight structures in \S\ref{swws} below. Yet we would like to recall right now (from \S2.7.2 of \cite{mymot})  that for any differential graded category $C$ there exists a unique "maximal" negative subcategory $C_-$ (it is not full unless $C$ is negative itself!) with $\obj C_=\obj C$; one sets 
$C_{-,i}(X,Y)$  to be zero  for $i>0$, to be equal to $C^i(X,Y)$ for $i<0$, and to be equal to $\ke \delta^0(C(X,Y))$ for $i=0$ (here $\delta^*$ is the corresponding differential).

To obtain a "more interesting" example for our theorem one should also "add" some morphisms of degree $1$. 
A certain problem here is that their compositions should be added as well, and then one should be careful to avoid non-zero $\cu$-morphisms between $\cp$ and $\cp[2]$. 
\end{rema}


\section{On the relationship to approximability}\label{sappr}

Now let us relate Theorem \ref{t1} to Definition 0.21 of \cite{neesat}. For this purpose we need a short reminder on $t$-structures. We start from an important existence statement, and recall the definitions later.

\begin{lem}\label{lts}
Let $\cp$ be a set of compact objects of (a smashing triangulated category) $\cu$. Then there exists a unique $t$-structure on $\cu$ (see Remark \ref{rts} below) such that $\cu^{t\le 0}=[\cup_{i\ge 0}\cp[i] \brab$.
\end{lem}
\begin{proof}
This is Theorem A.1 of \cite{talosa}.
\end{proof}

\begin{rema}\label{rts} 1. For the convenience of the reader (and to fix  notation), we recall the definition of a $t$-structure.

So, a couple of classes  $(\cu^{t\ge 0},\cu^{t\le 0})$ of objects of $ \cu$ is said to be a $t$-structure $t$ (on $\cu$)  if 
they satisfy the following conditions:

(i) $\cu^{t\ge 0} $ and $\cu^{t\le 0}$ are strict, i.e., contain all objects of $\cu$ isomorphic to their elements.

(ii) $\cu^{t\ge 0}\subset \cu^{t\ge 0}[1]$ and $\cu^{t\le
0}[1]\subset \cu^{t\le 0}$.

(iii)  $\cu^{t\le 0}[1]\perp \cu^{t\ge 0}$.

(iv) For any $M\in\obj \cu$ there exists a  
 distinguished triangle
$L_tM\to M\to R_tM{\to} L_tM[1]$
 such that $L_tM\in \cu^{t\le 0}, R_tM\in \cu^{t\ge 0}[-1]$.

2. Recall also that for $i\in \z$ the class $\cu^{t\le 0}[-i]$ (resp. $\cu^{t\ge 0}[-i]$) is denoted by $\cu^{t\le i}$ (resp. by $\cu^{t\ge i}$),
 and $\hrt$ is the full subcategory of $\cu$ whose object class is $\cu^{t=0}=\cu^{t\le 0}\cap \cu^{t\ge 0}$. An important property of $t$-structures is that the category $\hrt$ is abelian. 

3. Below we will mention the triangulated subcategory $\cu^b$ of $\cu$ whose object class equals $(\cup_{i\in \z}\cu^{t\le i})\cap (\cup_{i\in \z}\cu^{t\ge i})$. 
\end{rema}

\begin{coro}\label{capprox}
Let $\cp$ be a finite weakly negative set 
 of objects that compactly generates (a smashing triangulated category) $\cu$. 

1. Then $\cu$ is weakly approximable in the sense of \cite[Definition 0.21]{neesat}.

2. Moreover, if $\cp=\cup_{0\le i\le n}\cp_i$ (for some $n\ge 0$) and  $\cp_i\perp \cp_j[1]$ whenever $i\le j$, then $\cu$ is also approximable in the sense of loc. cit.
\end{coro}
\begin{proof} 1. Take $G=\bigoplus_{P\in \cp}P$. Then for $\cp'=\{G\}$ the class  $ [\cup_{i\in\z}\cp'[i]\brab$ is retraction-closed according to  Proposition 1.6.8 of \cite{neebook}. Hence  $ [\cup_{i\in\z}\cp'[i]\brab= [\cup_{i\in\z}\cp[i]\brab$; thus $\cu$ is compactly generated by $\cp'$ (as well). 

Now we should choose a $t$-structure. Lemma \ref{lts} gives the existence of $t$ such that  $\cu^{t\le 0}=[\cup_{i\ge 0}\cp'[i] \brab$.
Then $G\in \cu^{t\le 0}$, and Proposition \ref{pstar}(\ref{iortc}) implies 
 $G[-2]\perp  \cu^{t\le 0}$. Thus the couple $(G,t)$ fulfills the assumptions of \cite[Definition 0.21(i)]{neesat}.

Next, for any $F\in \cu^{t\le 0}$ Theorem \ref{t1}(1) gives the existence of a distinguished triangle $E\to F\to D\to E[1]$ such that $E\in [\cp' \brab$ and $D\in \cu^{t\le -1}$. Now, 
 $[\cp' \brab$ obviously lies in the class $\overline{\lan G\ra}^{[0,0]}$ (see Reminder 0.8(xii) of ibid.); hence the assumptions of \cite[Definition 0.21(ii)]{neesat} are fulfilled as well.

2. We take a distinguished triangle $E\to F\to D\to E[1]$ as above and check that $E$ belongs to the class $\overline{\lan G\ra}^{[0,0]}_{n+1}$ as defined in Reminder 0.8(xi) of ibid. (see Remark \ref{rex2}(3) below for more detail on these two definitions).

  $E$ belongs to $[\cp' \brab\subset [\cp \brab$. According to Theorem \ref{t1}(2), the latter class lies in $\cq_0\star \cq_1\star\dots \star\cq_n$, where $\cq_i$  consists of all coproducts of elements of $\cp_i$. Applying Proposition \ref{pstar}(\ref{istarkar}) $n$ times we obtain that 
	$\cq_0\star \cq_1\star\dots \star\cq_n\subset \kar_{\cu}(\cq^{n+1})$, 
	 where $\cq^{n+1}$ is the $n+1$-th $\star$-power of the class $\cq$ consisting of coproducts of (copies of) $G$. Thus $E$ belongs to the class $ \overline{\lan G\ra}^{[0,0]}_{n+1}$ as desired.\end{proof}

\begin{rema}\label{rex2}
1. Thus if  $\cp$ is finite in  the  examples mentioned in Remark \ref{rex2} then we obtain a (weakly) approximable category $\cu$.

2. In the case where $\cp$ consists of a single element $G$ the  assumptions of Corollary \ref{capprox}(2) just mean that $G\perp G[i]$ for all $i>0$; thus we obtain a justification of Remark 3.3 of ibid. We will discuss another possible argument for it in  Remark \ref{rstws}(\ref{iadjarg}) below. 

3. The ("additional") assumption of Corollary \ref{capprox}(2) is somewhat annoying; yet it cannot be dropped. 
 Indeed, let us describe an example for Corollary \ref{capprox}(1) such that the category $\cu$ is not approximable.

Take $\cu'=D(\ab)$, and set $\cu\subset \cu'$ to be the subcategory compactly generated by $\cp=\{G\}$, where $G=\z/p\z$ for  a prime number $p$. 
Then $(\cu,\cp)$ satisfy the assumptions of Corollary \ref{capprox}(1). 

Now let us check that $\cu$ is not approximable. Facts 0.22 of \cite{neesat} imply  (cf. Definition 0.14 of ibid.) that if our category $\cu$ is approximable then for $t$ being the $t$-structure such that $\cu^{t\le 0}=[\cup_{i\ge 0}\cp[i] \brab$ (cf. the proof of Corollary \ref{capprox}) there exists an integer $A>0$ 
 that satisfies the following condition: for any $F\in \cu^{t\le 0}$ there  exists a distinguished triangle $E\to F\to D\to E[1]$ such that  $D\in \cu^{t\le -1}$ and $E$ is a retract of an element of the class $\operatorname{Coprod}_A(G[-A,A])$. Here  
$\operatorname{Coprod}_A(G[-A,A])$ denotes the $A$-th $\star$-power of the class $G[-A,A]$ that consists of coproducts of (copies of) $G[s]$ for $-A\le s\le A$. 
In our case the key observation is that any element of $G[-A,A]$ is annihilated by the multiplication by $p$; hence any element of $\overline{\lan G\ra}^{[-A,A]}_{A}=\kar_{\cu}\operatorname{Coprod}_A(G[-A,A])$ is  annihilated by the multiplication by $p^A$. 

We take $F=\z/p^{A+1}\z$ (alternatively, one can take $F=
\bigoplus_{m\ge 0}(\z/p^m\z)$). It is easily seen that $F$ belongs to $\cu^{t\le 0}$. Let us prove that there cannot exist a distinguished triangle $E\to F\to D\to E[1]$ as above. Obviously, the complexes (of abelian groups) in $\cu^{t\le -1}$ have zero cohomology in non-negative degrees. Hence 
 the zero degree cohomology group $H^0(E)$ should surject onto $H^0(F)=\z/p^{A+1}\z$; yet this is impossible since  $E$ is  annihilated by the multiplication by $p^A$.

More generally, one can argue similarly in the following setting: start from the derived category $\cu'$ of left $R$-modules, where $R$ is a (not necessarily commutative) unital associative ring, choose $r\in R$ that is not left invertible and also is not a right zero divisor, take $G=R/Rr$, and set $\cu$ to be the subcategory of $\cu'$ compactly generated by $\cp=\{G\}$.

Moreover, these examples are closely related to the category $D_{qc,Z}(X)$ mentioned in Facts 0.23 of ibid.
\end{rema}


\section{On applications of approximability (a reminder after A. Neeman)}\label{sappl}

The writing of this note was motivated by the nice properties of approximable triangulated categories established in \cite{neesat}. 
So now we will say a few (more) words about them.

Firstly, it appears that the consequences of weak approximability are not as nice as that of approximability. However, in section 2 of ibid. several properties of weakly approximable categories are established; probably, some of them are quite useful.

So, we only recall that in Theorem 0.3 of ibid. an $R$-linear approximable triangulated category $\cu$ that is compactly generated by $\cp=\{G\}$ is considered, where $R$ is a commutative (unital associative) noetherian ring, and it is assumed that the $R$-module $\cu(G[i],G)$ is finitely generated for any $i\in \z$.\footnote{Moreover, it is also assumed that  $G[i]\perp G$ for $i\ll 0$; however, this assumption appears to follow from Facts 0.22(ii) (along with  Definition 0.21(i) of ibid.).} Denote by $\cu^c$ the (triangulated) subcategory of compact objects of $\cu$,  define the subcategory $\cu_c^-$ of $\cu$ via  Definition 0.16 of ibid., and take $\cu_c^b$ to be the subcategory whose object class equals $\obj\cu_c^-\cap \obj \cu^b$   (see Remark \ref{rts}(3) for the definition of $\cu^b$); here the corresponding $t$-structure $t$ is the one with $\cu^{t\le 0}=[\cup_{i\ge 0}\cp[i] \brab$. 

Then Theorem 0.3 of ibid. says that  the restricted Yoneda functor $\mathcal{Y}:\cu\to \adfur(\cu^c{}\opp,R-\modd)$ (the category of $R$-linear functors into $R$-modules) that sends an object $N$ of $\cu$ into the restriction of the functor it represents to $\cu^c$ gives a full embedding of $\cu_c^b$ into $\adfur(\cu^c{}\opp,R-\modd)$. 
Moreover, the image $\mathcal{Y}(\obj\cu_c^b)$ essentially consists of those cohomological functors that satisfy the following finiteness condition: for any $M\in \obj \cu^c$ the module $H(M[i])$ is finitely generated for all $i\in \z$ and vanishes for almost all values of $i$.\footnote{And one can easily check that finiteness is fulfilled (in this setting) whenever the $R$-module $\bigoplus_{i\in \z} H(G[i])$ is finitely generated.} Furthermore, the first of these statements is a consequence of more general (and rather interesting) assertions of the theorem.

As shown in several papers of Neeman (see  Corollary 0.4 and Remark 0.7 of ibid.) and also in \cite{bvtr}, Brown representability-type statements of this sort are useful for the construction of adjoint functors and $t$-structures, whereas various categories of quasi-coherent 
 sheaves over (rather general) schemes give interesting examples.

\section{On the relation to (weak) weight structures}\label{swws}

Let us start from the definition of a weight structure.

\begin{defi}\label{dwstr}

I. A couple $(\cu_{w\le 0},\cu_{w\ge 0})$ of classes  
 of objects of $\cu$ 
will be said to give a weight structure $w$ on 
$\cu$ if 
 the following conditions are fulfilled.

(i) $\cu_{w\le 0}$ and $\cu_{w\ge 0}$ are 
retraction-closed in $\cu$ (i.e., contain all $\cu$-retracts of their objects).

(ii) {\bf Semi-invariance with respect to translations.}

$\cu_{w\le 0}\subset \cu_{w\le 0}[1]$ and $\cu_{w\ge 0}[1]\subset
\cu_{w\ge 0}$.

(iii) {\bf Orthogonality.}

$\cu_{w\le 0}\perp \cu_{w\ge 0}[1]$.

(iv) {\bf Weight decompositions}.

 For any $M\in\obj \cu$ there
exists a distinguished triangle
$$L_wM\to M\to R_wM {\to} L_wM[1]$$
such that $L_wM\in \cu_{w\le 0} $ and $ R_wM\in \cu_{w\ge 0}[1]$.
\end{defi}

We will also need the following definitions.

\begin{defi}\label{dwso}
Let $i,j\in \z$; assume that a triangulated category $\cu$ is endowed with a weight structure $w$.

\begin{enumerate}
\item\label{idh}
The full category $\hw\subset \cu$ whose objects are
$\cu_{w=0}=\cu_{w\ge 0}\cap \cu_{w\le 0}$ 
 is called the {\it heart} of 
$w$.

\item\label{id=i}
 $\cu_{w\ge i}$ (resp. $\cu_{w\le i}$, resp. $\cu_{w= i}$) will denote the class $\cu_{w\ge 0}[i]$ (resp. $\cu_{w\le 0}[i]$, resp. $\cu_{w= 0}[i]$).

\item\label{idrest}
Let $\du$ be a full triangulated subcategory of $\cu$.

We will say that $w$ {\it restricts} to $\du$ whenever the couple $w_{\du}= (\cu_{w\le 0}\cap \obj \du,\ \cu_{w\ge 0}\cap \obj \du)$ is a weight structure on $\du$.

\end{enumerate}
\end{defi}

\begin{rema}\label{rstws}

\begin{enumerate}
\item\label{iwneg} Let us describe the relation of the results of this paper to weight structures; we will also mention "weak versions" of the latter below.

Firstly, the orthogonality axiom in Definition \ref{dwstr} immediately implies that the class $\cu_{w=0}$ is negative; hence any its subclass also is. More generally,  the class $\cu_{[0,1]}=\cu_{w\le 1}\cap \cu_{w\ge 0}$ is weakly negative; hence one can take any finite set $\cp$ of compact objects in it and pass to the subcategory of $\cu$ compactly generated by $\cp$. Moreover, recall that in this case $\cp$ consists of a set of cones of some $\hw$-morphisms (see Proposition 1.2.4(10) of \cite{bsnew}), and the orthogonality condition for $\cp_i$  in Theorem \ref{t1}(2) and Corollary  \ref{capprox}(2) can be checked (for the corresponding two-term complexes) in the homotopy category $K(\hw)$ (see Theorem 3.3.1(VII) of \cite{bws}; yet pay attention to part \ref{irconv} of this remark). 

Secondly, one can often "reconstruct" a weight structure from a given negative subclass of $\cu$; see Corollary 2.1.2 of \cite{bonspkar}, and Corollary 2.3.1 and Theorem 2.2.1 of \cite{bsnew}. Moreover, if $\cp$ is negative in the setting of Theorem \ref{t1} then the triangles provided by part 1 of that theorem are weight decompositions of elements of the corresponding $\cu_{w\ge 0}=\cu^{t\le 0}$ (cf. part \ref{iweak} of this remark); respectively, our argument is closely related to the one used for the proof of (part 1 of) loc. cit. 

\item\label{iwloc} Another reason why weight structures are relevant for the current paper is that there are several interesting methods for constructing them. Firstly, according to Theorem  8.2.3 of \cite{bws},  weight structures may be glued similarly to $t$-structures (see the seminal section 1.4 of \cite{bbd}).

Secondly, weight structures "often descend to localizations". Let us formulate the corresponding criterion: if $\cu$ is endowed with a weight structure $w$ and $\du$ is a full triangulated subcategory of $\cu$ then for the Verdier localization functor $\pi:\cu\to \cu'=\cu/\du$ one says that {\it $w$ descends to $\cu'$} if $w'=(\kar_{\cu'}(F(\cu_{w\le 0})), \kar_{\cu'}(F(\cu_{w\ge 0})))$ is a weight structure on $\cu$ (cf. Proposition 3.1.1(1) and Remark 3.1.2 of \cite{bsnew}), and Theorem  3.1.3(1) of ibid. says that this happens if and only if  for any $N\in \obj \du$ there exists a $\du$-distinguished triangle
$D(N)=(LN \stackrel{i}{\to} N\stackrel{j}{\to} RN\to LN[1])$
such that $i$ factors through (an element of) $\cu_{w\le 0}$ and $j$ factors through $\cu_{w\ge 0}$. Now, this condition is certainly fulfilled if $w$ restricts to $\du$, but the latter assumption is very far from being necessary. For instance, one can take $\du$ that is compactly generated by any set of (compact) elements of $\cu_{[0,1]}$; see part 3(ii) of loc. cit. (or Theorem 4.3.1.4(1) of \cite{bososn}). 

Another important statement on the existence of weight structures is Theorem 5 of \cite{paucomp}; yet  
  it appears to be useless for the context of the current paper. 


\item\label{irconv}
In 
 this note we use the ``homological convention'' for weight structures. This is the convention used by several papers of the first author (however, in \cite{bws} the so-called cohomological convention was used). Now, in the homological convention the functor $[1]$ "shifts weights" by $1$; note in contrast that (in the cohomological convention for $t$-structures that originates from \cite{bbd} and was used in \cite{neesat}) $[1]$ "shifts $t$-degrees" by $-1$.  

\item\label{iweak} Actually, the assumptions on $w$ in Definition \ref{dwstr} may be weakened (without loosing the actuality of this definition for the purposes of this note). 

Firstly, (the somewhat technical) axiom (i) may be replaced by a weaker assumption that $\cu_{w\le 0}$ and $\cu_{w\ge 0}$  are extension-closed (see. Proposition 1.2.4(4) of \cite{bsnew}).  We will use the symbol (i') for this alternative axiom. 

A more serious modification is to replace axiom  (iii) in Definition \ref{dwstr} by (i') along with the following axiom (iii'): $\cu_{w\le 0}\perp \cu_{w\ge 0}[2]$. We will say that a couple $w$ satisfying this set of axioms is a {\it weak weight structure};\footnote{Actually, axiom (i) is not necessary for the purposes of the current paper.} 
this definition is a particular case of \cite[Definition 3.11]{brelmot}  (that corresponds to the case $F=0$ in the notation of loc. cit.; see Remark 3.12(2) of ibid.). Obviously, the  heart of a weak weight structure (defined 
 as in Definition \ref{dwso}(\ref{idh})) is weakly negative; hence one can get certain examples of (weak) approximability from mixed complexes of $l$-adic \'etale sheaves (of weight $0$) over a variety $X_0$ over a finite  field (see Proposition 3.17(1) of ibid.). Moreover, our argument in Theorem \ref{t1}(1) 
 gives weak weight decompositions for elements of $\cu_{w\ge 0}$ for the following weak weight structure $w$: $\cu_{w\ge 0}=\kar_{\cu}([\cup_{i\ge 0}\cp[i]\brab)$ and $\cu_{w\ge 0}=\kar_{\cu}([\cup_{i\le 0}\cp[i]\brab)$; here $\cp$ is an arbitrary weakly negative compactly generating class of objects of $\cu$. 

Now, one can check that weak weight structures can be glued  similarly to weight structures. Moreover, it appears that weak weight structures descend to localizations 
and the additional condition in Theorem \ref{t1}(2) (and so also in  Corollary  \ref{capprox}(2))  passes to localizations as well under certain reasonable assumptions; however, one should study the details here. 


\item\label{iadjarg} As we have said above, our arguments are closely related to the weight structure $w$ and also to the $t$-structure $t$ such that  $\cu_{w\ge 0}=\cu^{t\le 0}=[\cup_{i\ge 0}\cp[i]\brab$; so, $w$ is {\it left adjacent} to $t$ (see Proposition 1.3.3 of \cite{bvtr} and  Definition 4.4.1 of \cite{bws}). Now, if one wants to argue in terms of $t$ (cf. Remark 3.3 of \cite{neesat}) then it is difficult to avoid the fact that $\hrt$ is naturally equivalent to the category $\adfu(\hu,\ab)$, 
 where $\hu\subset \cu$ 
 and $\obj \hu=\cp$; see Theorem 4.5.2(II.2) of \cite{bws}  or Theorem 3.3.1(II.3) of \cite{bvtr}.\footnote{Actually, to apply these theorems one should note that this functor category essentially will not change if one replaces $\hu$ by its additive hull or by $\kar_{\cu}\hu$. Alternatively, one can easily deduce the statement in question from   Theorem 3.2.2(6) of \cite{bwcp}.} However, if $\cp$ is finite (recall that this is the case for  Corollary  \ref{capprox}) then one may also apply Theorem 1.3 of \cite{hoshino}; this appears to be the only way to justify (the general case of) Remark 3.3 of \cite{neesat} without mentioning weight structures. 

 Still the authors do not know how to extend this "adjacent" argument to weak weight structures (that correspond to part 1 of Corollary  \ref{capprox}(1) and also to the case $n>0$ of its part 2). Note also that if the class $\cp$ in Theorem \ref{t1} is negative but not small 
then it is not clear whether the corresponding adjacent $t$-structure exists (whereas  $w$ as above exists according to Corollary 2.3.1  of \cite{bsnew}; cf. also Remark 2.3.2(4) of ibid.). 


\end{enumerate}

\end{rema}

\end{document}